\documentclass[12pt]{amsart}
\usepackage{amssymb}
\usepackage{mathtools}      
\usepackage{bm}             
\usepackage{amsthm}
\usepackage{thmtools}
\usepackage{enumitem}       
\usepackage[colorlinks=false,backref=page]{hyperref} 
\usepackage{tikz-cd}        
\usepackage{tikz}
\usetikzlibrary{datavisualization.formats.functions}
\usepackage{algorithm}
\usepackage{algpseudocode}

\usepackage{graphicx}

\theoremstyle{plain}
\newtheorem{theorem}{Theorem}[section]
\newtheorem{lemma}[theorem]{Lemma}
\newtheorem{proposition}[theorem]{Proposition}

\theoremstyle{definition}

\newtheorem{remark}[theorem]{Remark}
\newtheorem{example}[theorem]{Example}

\newcommand{\R}{\mathbb{R}}
\newcommand{\Z}{\mathbb{Z}}

\newcommand{\N}{\mathbb{N}}

\renewcommand{\phi}{\varphi}
\renewcommand{\kappa}{\varkappa}
\newcommand{\cC}{\mathcal{C}}

\begin{document}
\title[Gelfond exponent of weighted Thu-Morse sequence]{A rigorous computer aided estimation for Gelfond exponent of weighted Thue-Morse sequences}

%

%

\author{Yiwei Zhang, Ke Yin}
\address[Yiwei Zhang, Ke Yin]{School of Mathematics and Statistics\\
  Center for Mathematical Science\\
  Hubei Key Laboratory of Engineering Modeling and Scientific Computing\\
  Huazhong University of Science and Technology\\
  1037 Luoyu Road, Wuhan, China 430074}
\email[Yiwei Zhang]{yiweizhang831129@gmail.com}

\author{Wanquan Wu}
\address[Wanquan Wu]{School of Computer Science and Technology\\
  Huazhong University of Science and Technology \\
  1037 Luoyu Road, Wuhan, China 430074}


\thanks{We would like to thank Prof. Aihua Fan for invaluable discussions, particularly for his kindness for introducing us in detail about the definition and the background of Gelfond exponent; and its relationships with ergodic optimization problems, i.e., Theorem \ref{theo_gelfondandergodic} and the proof stated in Section \ref{sec_gelfondergodic}.  Mathematical materials for the background in Section \ref{sec_introduction} and Theorem \ref{theo_gelfondandergodic} in Section \ref{sec_gelfondergodic} are contributed by Prof. Aihua Fan, subject to our own understanding on writing.}

\begin{abstract}
In this paper, we will provide a mathematically rigorous computer aided estimation for the exact values and robustness for Gelfond exponent of weighted Thue-Morse sequences. This result improves previous discussions on Gelfond exponent by Gelfond, Devenport, Mauduit, Rivat, S\'{a}rk\"{o}zy and Fan et. al.
\end{abstract}
\maketitle
\section{Introduction}\label{sec_introduction}

The weighted $(c)$-Thue-Morse sequence was first introduced in \cite{Fan17}, and is among the simplest and typical multiplicative sequence, and attracts great interest from various mathematical and computational sciences.  For every real number $c\in[0,1)$, \emph{the weighted $(c)$-Thue-Morse sequence} is described by the formula
$$
t^{(c)}(n):=e^{2\pi i cs(n)},~~\forall n\in \N,
$$
where $s(n)$ is the sum of digits of $n$ based 2. In particular, $t(n):=t^{(1/2)}(n)$ is the classical Thue-Morse sequence. As a time series, this sequence is hybrid, in the sense that: on one hand, the subward complexity grows linearly, while on the other hand, there are various ways in which it can be construed as pseudorandom.

One of the studies on characterizing the pseudorandomness of the weighted Thue-Morse sequence is the study on its \emph{Gelfond type of oscillations}. To be more precise, that is to study the existence of a constant $\alpha^{(c)}\in (1/2,1)$, such that
\begin{equation}\label{equ_Gelfond type oscillations}
    \max_{0\leq t\leq 1}\left|\sum_{n=0}^{N-1}t^{(c)}(n)e^{2\pi int}\right|=O(N^{\alpha^{(c)}}).
\end{equation}
If $\alpha^{(c)}$ exists, then we say $t^{(c)}(n)$ is of \emph{Gelfond type} and the smallest $\alpha^{(c)}$ will be denoted by \emph{Gelfond exponent} $\Delta^{(c)}$.

Gelfond in \cite{Gelfond68} showed that $\Delta^{(1/2)}=\frac{\log 3}{\log4}$. Using this result, Mauduit and S\'{a}rk\"{o}zy \cite{MS98} showed that the classical Thue-Morse sequence is highly uniformly distributed, in the sense that for positive integers $a,b,M,N$ with $a(M-1)+b<N$, then
$$
\sum_{n=0}^{M-1}t(an+b)=O(N^{\log3/\log4}).
$$
On the other hand, Fan \cite{Fan16,Fan17} used Gelfond exponent to give a quantitative estimation on the growth size of the weighted Birkhoff ergodic sum. That is, for a measure preserving map $(X,\mathcal{B},\nu,T)$, then for every $f\in L^{2}(\nu)$, and every $\delta>0$, we have
$$
\nu\mbox{-a.e.-}x,~~~\sum_{n=0}^{N-1}t^{(c)}(n)f(T^{n}(x))=O(N^{\Delta^{(c)}}\log^{2}N\log^{1+\delta}\log N).
$$
Such estimation is usually treated as a probabilistic comparison on the weighted Birkhoff ergodic sum for the orthogonality relationships between topological oscillations of the sequences and zero topological entropy or uniquely ergodic dynamical systems \cite{Fan16,Fan17,FJ17}. For example, Sarnak's Conjecture for the M\"{o}bius sequence \cite{Sarnak09} and Wiener-Winter theory \cite{Fan172,WW41}.

Mauduit, Rivat and S\'{a}rk\"{o}zy \cite{MRS16} recently gave an elegant proof showing that every $t^{(c)}(n)$ is of Gelfond type, and moreover, they show that
$$
\Delta^{(c)}\leq 1-\frac{\pi^2}{20\log2}||c||^{2},~~\mbox{where}~~||c||:=\inf_{z\in\Z}|c-z|.
$$
Unfortunately, this estimate is not optimal. In addition, Konieczny \cite{Konieczny17} linked the study on Gelfond bound with the Gowers uniformity norm, but their expression on Gelfond bound is also implicit. Recently, we also remark that Fan,Shen \cite{FS} developed Davenport's idea and gave exact values of Gelfond exponent for several other special parameters including $c=1/4,3/4$. They ask whether there exists a universal method to estimate the exact value of Gelfond exponent for arbitrary $c\in[0,1)$.

In this paper, we would like to develop a computational-aided estimation on the exact value of the Gelfond exponent for general $c$-Thue Morse sequences. In particular, our approach enables us to test $c\in \Lambda_{L}=\{ \frac{i}{2^L}: L\in\N \text{ and } 0\leq i\leq 2^{L}-1\}$. For example, when $L=10$, and letting
$$
 \Lambda_{10}^{U}= \left\{\frac{192}{1024},\frac{309}{1024},\frac{390}{1024},\frac{391}{1024},\frac{832}{1024},\frac{715}{1024},\frac{634}{1024},\frac{633}{1024}\right\},
$$
we can
\begin{itemize}
  \item estimate the accurate value of $\Delta^{(c)}$ for every $c\in \Lambda_{10}\backslash\Lambda_{10}^{U};$
  \item show that the Gelfond exponent function $c\to \Delta^{(c)}$ is real analytic for every $c\in\Lambda_{10}\backslash\Lambda_{10}^{U}$.
\end{itemize}
The graph of the Gelfond exponent function is illustrated in Figure \ref{fig:butterfly}, and the explanations of the above two assertions will be given in Section \ref{sec_output}.

\begin{figure}
\includegraphics[width=0.8\textwidth]{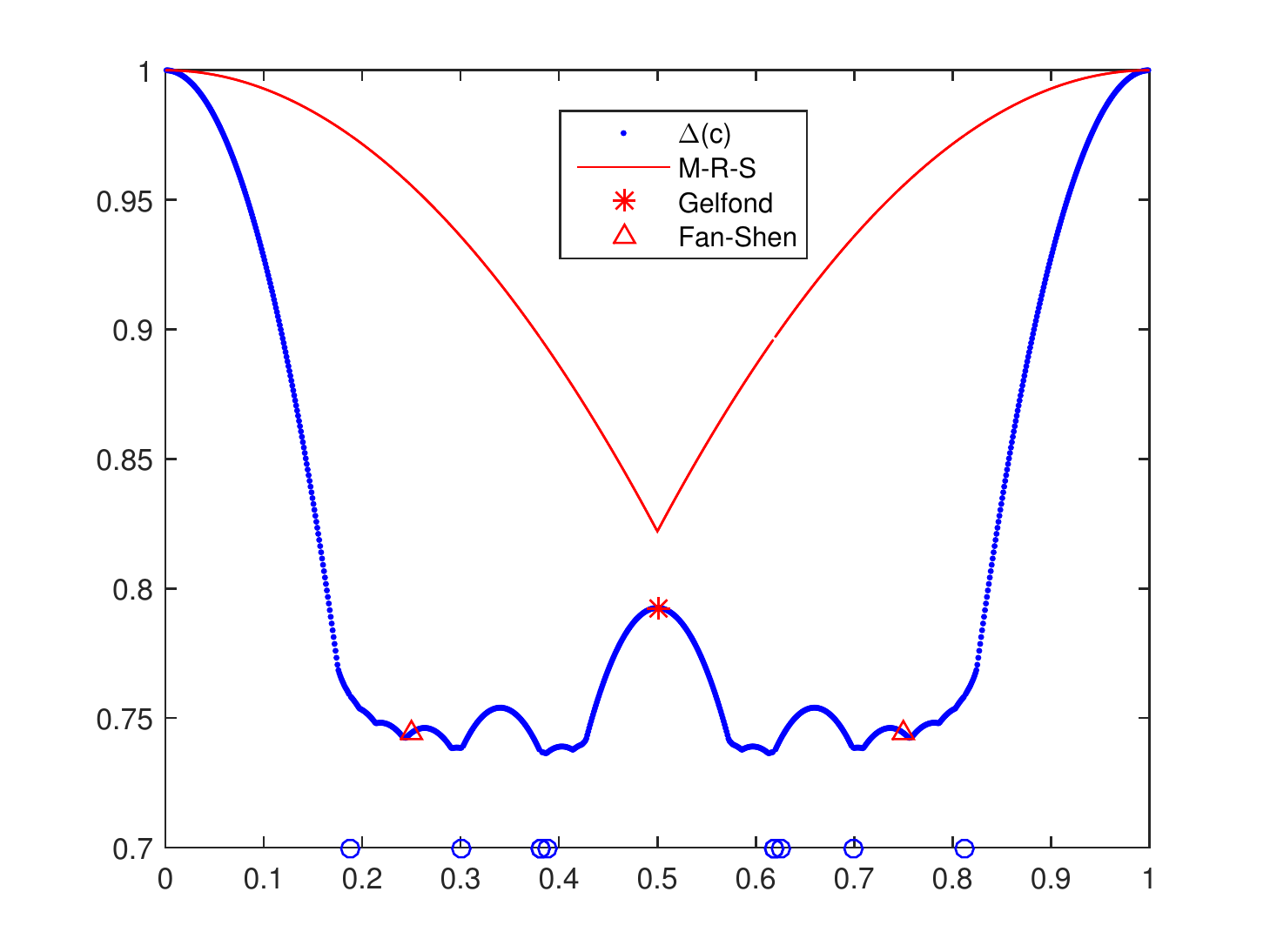}
\caption{The blue dot line is our estimation for the exact value of Gelfond exponent. This dot line is strictly smaller than the bound obtained by Mauduit, Rivat and S\'{a}rk\"{o}zy for any $c\in[0,1)$ (as drawn by the normal red line), and coincides with Gelfond's estimation at $c=1/2$(the red star), and Fan-Shen's estimation at $c=1/4,3/4$ (the red triangles). Followed by our estimation, it is observed that, differing to the Mauduit, Rivat and S\'{a}rk\"{o}zy bound curve, the Gelfond exponent function $c\to \Delta^{(c)}$ is symmetric at point $1/2$, but is not globally concave at each half. Instead, the graph of such function seems like a zigzag curve with many ``bubbles''(of various sizes fluctuations), and many real analytic points. The points in $\Lambda_{10}^{U}$ are drawn as the blue circles.}
\label{fig:butterfly}
\end{figure}

Our methodology contains a combination of two ingredients. Firstly, we state an equivalence between the estimation of the exact value of Gelfond exponent and an ergodic optimization problem for the doubling map with the upper semi-continuous $c$-parameterized potential formalised by $\log\cos\pi(x+c)$, see Theorem \ref{theo_gelfondandergodic} (the latter is to find the invariant measure maximizing the integral of the potential). Secondly, we will use the method developed in \cite{BZ16} to reduce the ergodic optimization problem to a finite dimensional combinatorics problem (i.e., maximizing mean cycle problem) on a sequence of edge weighted De Brujin graphs see Theorem \ref{theo_gapcondtion}. Based on these two ingredients, we implement an algorithm (i.e., Algorithm \ref{alg:gap_criterion}) to find the unique maximizing periodic invariant measure, and thus estimate the exact value of the Gelfond exponent. 
Our computational method provides a rigorous proof of the existence of $\Delta^{(c)}$ by checking an inequality \eqref{eq:gap_inequality} for a finite number of cases using a unified algorithm that always terminates in definite number of steps. The number of cases checked depends on given computational resources and machine error, and our case study for the parameters in $\Lambda_{10}$, shows that most of the values $c$ can be successfully checked within a few cases, while verifying for the rest of $c$ is beyond out computational power. As far as we are concerned, this is the first attempt to investigate the Gelfond exponent via combinatorial and computer aided approach.

The paper is organized as follows. We will first explain the mathematical principles of our Algorithm \ref{alg:gap_criterion} in Section \ref{sec_gelfondergodic} and Section \ref{Sec_combinatorial}. The implementation of Algorithm \ref{alg:gap_criterion} is provided in Section \ref{Sec_verification}, and the outputs, as well as some discussions for the further studies are provided in Section \ref{sec_output}. Some supplementaries on Howard algorithm are provided in the Appendix \ref{sec_appendix} for the convenience of readers.

\section{Gelfond exponent and ergodic optimization}\label{sec_gelfondergodic}

This section is devoted to describing the equivalence between the estimation of the Gelfond exponent and an ergodic optimization problem. Before stating the main result, let us recall some basic notions from ergodic optimization theory. Let $(X,d)$ be a compact metric space, and $f \colon X \to X$ be a continuous transformation.
Let $\mbox{Inv}(f)$ be the set of all $f$-invariant probability measures.
If $\mu \in \mbox{Inv}(f)$ is supported on a periodic orbit, then it is called a \emph{periodic measure}.

Given an upper semi-continuous potential function $\varphi \colon X \to \R$,
the \emph{ergodic supremum} of $f$ is defined by
$$
\mbox{ergsup} (f,\varphi) := \sup_{\mu \in Inv(f)} \int \varphi d\mu \, ,
$$
If the $\sup$ is attained at a $\mu \in \mbox{Inv}(f)$ then we say that
the measure $\mu$ is \emph{maximizing} for~$f$. Such maximizing measures always exist, due to the compactness of $X$, and the semi-continuity of $\varphi$.

The following two propositions will be useful later on.
\begin{proposition}\cite[Prop2.2]{Je00}\label{prop_ergequivalence}
\begin{eqnarray}
  \mbox{ergsup}(f,\varphi) &=& \sup_{x\in X}\limsup_{n\to\infty}\frac{1}{n}\sum_{i=0}^{n-1}\varphi\circ f^{i}(x) \\
   &=& \limsup_{n\to\infty}\frac{1}{n}\sup_{x\in X}\sum_{i=0}^{n-1}\varphi\circ f^{i}(x)<+\infty.
\end{eqnarray}
\end{proposition}

\medskip

\begin{proposition}\label{lem_monotonemaximizing}
  Suppose two upper semi-continuous potential functions $\varphi_{1},\varphi_{2}:X\to \R$ satisfies:
\begin{itemize}
  \item $\varphi_{1}\leq \varphi_{2};$
  \item $\mu$ is the unique maximizing measure for $\varphi_{2}$, and
  $$
  \varphi_{1}(x)=\varphi_{2}(x),~~~\forall x\in\mbox{supp}(\mu),
  $$
\end{itemize}
then $\mu$ is also the unique maximizing measure for $\varphi_{1}$.
\end{proposition}
\begin{proof}
Under the hypothesis, for every invariant measure $\nu\neq\mu$, we have
$$
\int\varphi_{1}d\nu\leq\int\varphi_{2}d\nu<\int\varphi_{2}d\mu=\int\varphi_{1}d\mu.
$$
Therefore, $\mu$ is the unique maximizing measure for $\varphi_{1}$, as required.
\end{proof}

\bigskip

In the context of our present work, we will particularly concentrate on $X$ being a one dimensional torus $\mathcal{S}$, and $f(x):=2x(\mod 1),$ and $g^{(c)}(x):=\log|\cos\pi(x+c)|,~~~~~~x\in \mathcal{S},~~c\in[0,1)$. Denote by
$\beta^{(c)}:=\mbox{ergsup}(f,g^{(c)})$, and the main result in this section is as follows.
\begin{theorem}\label{theo_gelfondandergodic}
\begin{equation}\label{equ_Gelfondexponentandergodicoptimization}
    \Delta^{(c)}=1+\frac{\beta^{(c)}}{\log 2},~~\forall c\in(0,1).
\end{equation}
\end{theorem}
The proof essentially follows from Mauduit, Rivat and S\'{a}rk\"{o}zy \cite{MRS16} and Fan \cite{Fan17}. Though any
specialists in Gelfond exponent shouldn't have any difficulty in providing those
details themselves, we decide to write down the details here for the convenience of readers, as we weren't able to find any precise reference, and the proof itself is ingenious.
\begin{proof}
For each $c\in(0,1)$, denote by
\begin{align*}
S_{N}(x)
&=:\sum_{n=0}^{N-1}t^{(c)}(n)e^{2\pi i x}\\
&=\sum_{n=0}^{N-1}e^{2\pi icS(n)+nx}=\sum_{n=0}^{N-1} F_{n}(x),~~~\mbox{where}~~F_{n}(x):=e^{2\pi i(cS(n)+nx)}.
\end{align*}
First, we claim that
\begin{equation}\label{equ_S2m}
    S_{2^{m}}(x)=\prod_{j=0}^{m-1}(F_{0}(x)+F_{2^{j}}(x))~~\forall m\in \N.
\end{equation}
We will proceed the proof of \eqref{equ_S2m} by induction on $m$. By the definition of $s(n)$, it is clear that
$$
F_{i+2^{j}}(x)=F_{i}(x)\cdot F_{2^{j}}(x),~~\forall j\in\N,~~\mbox{and}~~0\leq i\leq 2^{j}-1.
$$
In particular,
$$
F_{0}(x)\equiv1.
$$
Therefore,
\begin{itemize}
  \item when $m=1$, then
   \begin{align*}
S_{2}(x)&=\sum_{n=0}^{1}e^{2\pi i(cs(n)+nx)}=F_{0}(x)+F_{1}(x).
\end{align*}
  \item when $m=2$, then
   \begin{align*}
S_{4}(x)&=\sum_{n=0}^{3}e^{2\pi i(cs(n)+nx)}\\
&=F_{0}(x)+F_{1}(x)+F_{2}(x)+F_{3}(x)\\
&=F^{2}_{0}(x)+F_{0}(x)F_{1}(x)+F_{0}(x)F_{2}(x)+F_{1}(x)F_{2}(x)\\
&=(F_{0}(x)+F_{1}(x))(F_{0}(x)+F_{2}(x)).
\end{align*}
  \item Suppose $S_{2^{m}}(x)=\prod_{j=0}^{m-1}(F_{0}(x)+F_{2^{j}}(x))$,
  then
\begin{align*}
S_{2^{m+1}}(x)&=S_{2^{m}}(x)+\sum_{i=0}^{2^{m}-1}F_{i+2^{m}}(x)\\
&=S_{2^{m}}(x)+F_{2^{m}}(x)\sum_{i=0}^{2^{m}-1}F_{i}(x)~~~~\text{(By 2-multiplicity)}\\
&=S_{2^{m}}(x)\cdot(F_{0}(x)+F_{2^{m}}(x))\\
&=\prod_{j=0}^{m}(F_{0}(x)+F_{2^{j}}(x))~~~\text{(Using induction hypothesis)}.
\end{align*}
\end{itemize}
This completes the proof of \eqref{equ_S2m}.

Therefore,
\begin{align*}
|S_{2^{m}}(x)|&=\prod_{j=0}^{m-1}\left|F_{0}(x)+F_{2^{j}}(x)\right|=\prod_{j=0}^{m-1}\left|1+e^{2\pi i(cs(2^{j})+2^{j}x)}\right|\\
&=\prod_{j=0}^{m-1}\left|1+e^{2\pi i(c+2^{j}x)}\right|=2^{m}\prod_{j=0}^{m-1}\left|\cos\pi(c+2^{j}x)\right|\cdot\left|e^{i\pi(c+2^{j}x)}\right|\\
&=2^{m}\prod_{j=0}^{m-1}\left|\cos\pi(c+2^{j}x)\right|.
\end{align*}
Note that since the function $g^{(c)}(x)=\log|\cos\pi(x+c)|$ is upper semi-continuous, it follows from Proposition \ref{prop_ergequivalence} that
\begin{equation}\label{equ_ergodicequivalence}
\limsup_{m\to\infty}\frac{1}{m}\sup_{x\in[0,1]}\sum_{j=0}^{m-1}\log|\cos\pi(c+2^{j}x)|=\sup_{\mu\in Inv(g)}\int\log|\cos\pi(x+c)|d\mu=\beta^{(c)}.
\end{equation}
Thus
\begin{equation}\label{equ_ergodicsubsequence}
\limsup_{m\to\infty}\sup_{x\in[0,1]}|S_{2^{m}}(x)|=(2^{m})^{1+\frac{\beta^{(c)}}{\log2}},
\end{equation}
Equation \eqref{equ_ergodicsubsequence}, together with \cite[Theo5]{Fan17} yields that there exists a constant $D>0$ such that
$$
\max_{x\in[0,1]}\left|S_{n}(x)-S_{m}(x)\right|\leq D(n-m)^{1+\frac{\beta^{(c)}}{\log2}}.
$$
That is
$$
\Delta^{(c)}=1+\frac{\beta^{(c)}}{\log2},
$$
as was to be proved.
\end{proof}
\begin{remark}
In particular, when $c=1/2$, we will show in Example \ref{example:classicalthuemorse} in Section \ref{sec_output} that the periodic measure $\frac{1}{2}(\delta_{1/3}+\delta_{2/3})$ is the unique $g_{1/2}$-maximizing measure. This implies that
$$
\beta^{(1/2)}=\frac{1}{2}\left(\log|\cos\pi(1/3+1/2)|+\log|\cos\pi(2/3+1/2)|\right)=\log\frac{\sqrt{3}}{2}.
$$
So we reobtain the Gelfond exponent for the classical Thue Morse sequence
$$
\Delta^{(1/2)}=1+\frac{\log\frac{\sqrt{3}}{2}}{\log2}=\frac{\log3}{\log4}.
$$
\end{remark}

\section{Combinatorial optimization truncation}\label{Sec_combinatorial}

This section is aiming to convert the above ergodic optimization problem into a ``limit state'' of a finite dimensional combinatorial optimization problem (i.e., maximum mean cycle problem on a sequence of quotient de Bruijn graphs). Some preliminaries in graph, wavelet and combinatorial optimization theory are provided, and the main theorem (Theorem \ref{theo_gapcondtion}) in this section is given afterwards.

Let $\Omega=\{0,1\}^{\N}$ and $\Omega^*=\bigcup_{i=0}^{\infty}\{0,1\}^{i}$ with convention $\{0,1\}^{0}=\emptyset$. Given a word $\omega\in\Omega^*$, denote by cylinder $[\omega]$, the set of elements of $\Omega$ that have $\omega$ as the initial sub-word.

\subsection{Quotient de Bruijn graphs and periodic measures}\label{sec_debrujingraph}

The concept of \emph{de Bruijn graphs} $\widetilde{G}_{n}=(\widetilde{V}_{n},\widetilde{E}_{n})_{n\geq 1}$ (together with their analogues for larger alphabets) was introduced independently by De Bruijn \cite{de} and Good \cite{Go}, and is defined as follows.
\begin{itemize}
\item every node $v\in V_{n}$ is exactly the words of length $n-1$ in the alphabet $\{0,1\}$;
\item every edge $e\in E_{n}$ is exactly the words of length $n$ in the alphabet $\{0,1\}$;
\item for each word $\omega$ of length $n$, the source node and the target node of the arc $\omega$ are respectively its initial and final subwords of length $n-1$.
\end{itemize}
The first five de Bruijn graphs are pictured in Figure~\ref{f.graphs}.

\begin{figure}[htb]
\begin{tikzpicture}[baseline] 
   \node(emptyword) at (0, 0){$\emptyset$};
   \draw(emptyword)edge[loop above]  node{\small $0$} (emptyword);
   \draw(emptyword)edge[loop below] node{\small $1$} (emptyword);
\end{tikzpicture}
\qquad
\begin{tikzpicture}[baseline] 
   \node(0) at (0,.75){\small $0$};
   \node(1) at (0,-.75){\small $1$};
   \draw(0)edge[loop above] node{\footnotesize $00$} (0);
   \draw(0)edge[->,bend left=15] node[right]{\footnotesize $01$} (1);
   \draw(1)edge[loop below] node{\footnotesize $11$} (1);
   \draw(1)edge[->,bend left=15] node[left]{\footnotesize $10$} (0);
\end{tikzpicture}
\qquad
\begin{tikzpicture}[baseline] 
   \node(00) at (  0, 1.8){\footnotesize $00$};
   \node(01) at ( .8,   0){\footnotesize $01$};
   \node(10) at (-.8,   0){\footnotesize $10$};
   \node(11) at (  0,-1.8){\footnotesize $11$};
   \draw(00)edge[loop above]       node{\scriptsize $000$} (00);
   \draw(00)edge[->]               node[right]{\scriptsize $001$} (01);
   \draw(01)edge[->,bend right=15] node[above]{\scriptsize $010$} (10);
   \draw(01)edge[->]               node[right]{\scriptsize $011$} (11);
   \draw(10)edge[->]               node[left] {\scriptsize $100$} (00);
   \draw(10)edge[->,bend right=15] node[below]{\scriptsize $101$} (01);
   \draw(11)edge[->]               node[left] {\scriptsize $110$} (10);
   \draw(11)edge[loop below]       node{\scriptsize $111$} (11);
\end{tikzpicture}
\qquad
\begin{tikzpicture}[baseline] 
   \node(000) at (  0, 2.2){\scriptsize $000$};
   \node(001) at ( .8, 1.5){\scriptsize $001$};
   \node(100) at (-.8, 1.5){\scriptsize $100$};
   \node(010) at (  0,  .6){\scriptsize $010$};
   \node(101) at (  0, -.6){\scriptsize $101$};
   \node(011) at ( .8,-1.5){\scriptsize $011$};
   \node(110) at (-.8,-1.5){\scriptsize $110$};
   \node(111) at (  0,-2.2){\scriptsize $111$};
   \draw(000)edge[loop above]      (000);
   \draw(000)edge[->]              (001);
   \draw(001)edge[->]              (010);
   \draw(001)edge[->]              (011);
   \draw(010)edge[->]              (100);
   \draw(010)edge[->,bend left=15] (101);
   \draw(011)edge[->]              (110);
   \draw(011)edge[->]              (111);
   \draw(100)edge[->]              (000);
   \draw(100)edge[->]              (001);
   \draw(101)edge[->,bend left=15] (010);
   \draw(101)edge[->]              (011);
   \draw(110)edge[->]              (100);
   \draw(110)edge[->]              (101);
   \draw(111)edge[->]              (110);
   \draw(111)edge[loop below]      (111);
\end{tikzpicture}
 \quad
 \begin{tikzpicture}[baseline,>={Stealth[round]}]
   \node(0000) at (   0, 3.5){\tiny $0000$};
   \node(0001) at ( 1.3, 3  ){\tiny $0001$};
   \node(1000) at (-1.3, 3  ){\tiny $1000$};
   \node(1001) at (   0, 2.5){\tiny $1001$};
   \node(0010) at (   1, 1  ){\tiny $0010$};
   \node(0100) at (  -1, 1  ){\tiny $0100$};
   \node(0011) at ( 1.9, 0  ){\tiny $0011$};
   \node(0101) at ( .85, 0  ){\tiny $0101$};
   \node(1010) at (-.85, 0  ){\tiny $1010$};
   \node(1100) at (-1.9, 0  ){\tiny $1100$};
   \node(1011) at (   1,-1  ){\tiny $1011$};
   \node(1101) at (  -1,-1  ){\tiny $1101$};
   \node(0110) at (  0 ,-2.5){\tiny $0110$};
   \node(0111) at ( 1.3,-3  ){\tiny $0111$};
   \node(1110) at (-1.3,-3  ){\tiny $1110$};
   \node(1111) at ( 0  ,-3.5){\tiny $1111$};

   \draw(0110)edge[->,bend left=17] (1100);
   \draw(1001)edge[->,bend left=17] (0011);
   \draw(1100)edge[->,bend left=17] (1001);
   \draw(0011)edge[->,bend left=17] (0110);

   \draw(0001)edge[line width=6pt,draw=white] (0010);
   \draw(0100)edge[line width=6pt,draw=white] (1000);
   \draw(1011)edge[line width=6pt,draw=white] (0111);
   \draw(1110)edge[line width=6pt,draw=white] (1101);

   \draw(0000)edge[loop above]      (0000);
   \draw(0000)edge[->]              (0001);
   \draw(0001)edge[->]              (0010); 
   \draw(0001)edge[->,bend left=12] (0011);
   \draw(0010)edge[->]              (0100);
   \draw(0010)edge[->]              (0101);
   \draw(0011)edge[->,bend left=12]  (0111);
   \draw(0100)edge[->]              (1000); 
   \draw(0100)edge[->]              (1001);
   \draw(0101)edge[->,bend right=20](1010);
   \draw(0101)edge[->]              (1011);
   \draw(0110)edge[->]              (1101);
   \draw(0111)edge[->]              (1110);
   \draw(0111)edge[->]              (1111);
   \draw(1000)edge[->]              (0000);
   \draw(1000)edge[->]              (0001);
   \draw(1001)edge[->]              (0010);
   \draw(1010)edge[->]              (0100);
   \draw(1010)edge[->,bend right=20](0101);
   \draw(1011)edge[->]              (0110);
   \draw(1011)edge[->]              (0111); 
   \draw(1100)edge[->,bend left=12] (1000);
   \draw(1101)edge[->]              (1010);
   \draw(1101)edge[->]              (1011);
   \draw(1110)edge[->,bend left=12] (1100);
   \draw(1110)edge[->]              (1101); 
   \draw(1111)edge[->]              (1110);
   \draw(1111)edge[loop below]      (1111);
 \end{tikzpicture}
\caption{Small order de Bruijn graphs $\widetilde{G}_n$ for $1 \le n \le 5$.}\label{f.graphs}
\end{figure}
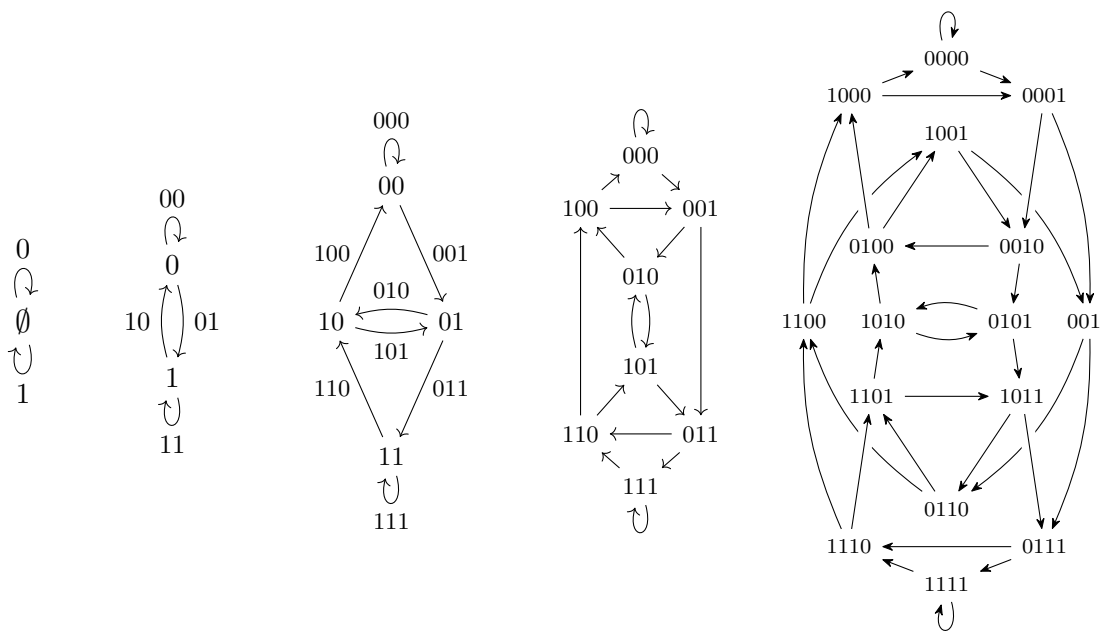

Each simple cycle $\mathcal{C}\subset \widetilde{G}_{n}$ canonically associates a periodic infinite word in the symbolic space $\Omega$ by the concatenation of the words associating from the edges. For convenience, we represent such sequence by the finite repeated block. For example, the symbolic representation $01010101......$ is simply abbreviated as $\overline{01}$. In fact, we will use this simple block to represent the whole period orbit by repeating the binary expansion (of the symbolic representation) to the dyadic fraction of the points and the iterations of the doubling mapping. For example, the cycle represented by $\overline{01}$ gives the orbits
$$
p=2, ~~~~~\frac{1}{3}\rightleftharpoons\frac{2}{3},
$$
and the cycle represented by $\overline{001}$ gives the orbit
$$
p=3, ~~~~~\frac{1}{7}\rightarrow\frac{2}{7}\rightarrow\frac{4}{7}\rightarrow\frac{1}{7}.
$$

With this convention, denote by $G_{n}=(V_{n},E_{n})$ the \emph{quotient de Bruijn graphs} with
\[
V_{n}\ni[v]:=\left\{\begin{array}{ll}
                     v, & \mbox{if}~ v\notin\{\underbrace{0\cdots 0}_{n},\underbrace{1\cdots1}_{n}\} \\
                    \{\underbrace{0\cdots 0}_{n},\underbrace{1\cdots1}_{n}\}, & \mbox{otherwise}.
                   \end{array}\right.
\]
and
$$
E_{n}:=\left\{([u]\rightarrow[v]),~\mbox{if}~(u\rightarrow v)\in \widetilde{E}_{n}\right\}.
$$
In informal terms, the quotient de Bruijn graph $G_{n}$ could be viewed as the de Bruijn graph $\widetilde{G}_{n}$ gluing at two self loops. Analogous to de Bruijn graph, every simple cycle in $G_{n}$ admits a unique periodic orbit, subject to identifying the orbits $\bar{0}$ and $\bar{1}$. Due to this fact, readers might realize later that the quotient de Bruijn graph is a more appropriate tool (e.g., satisfying Theorem \ref{theo_gapcondtion}) than the de Bruijn graph for dealing with the ergodic optimization problems for the potentials on the torus.

\subsection{Association between edge weights on Quotient de Bruijn graphs and Haar function}

{\emph Haar function} is defined by
$$
h_{\omega}:=\frac{1}{2}(\chi_{[\omega0]}-\chi_{[\omega1}]),
$$
where $\chi$ denotes characteristic function, and $\omega0,~\omega1$ means concatenated word.

It is easy to see the set $\{1\}\bigcup\{h_{\omega},\omega\in\Omega^*\}$ forms an orthonormal wavelet basis of $L^{2}(\Omega)$, i.e., the Hilbert space of functions that are square integrable with respect to Lebesgue measure. Thus, every $\varphi\in C^{0}(\Omega)\subseteq L^{2}(\Omega)$ can be uniquely and pointwisely represented as a Haar series:
$$
\varphi=c(\varphi)+\sum_{\omega\in\Omega^*}c_{\omega}(\varphi)h_{\omega},
$$
where the Haar coefficient is
$$
c(\varphi):=\int \varphi dx,~~\mbox{and}~~c_{\omega}(\varphi):=2^{|\omega|+2}\int \varphi h_{\omega}dx=2^{|\omega|+1}(\int_{[\omega0]}\varphi dx-\int_{[\omega1]}\varphi dx).
$$
The n-th approximation of $\varphi$ is given by the sum of the truncated Haar series, namely,
$$
A_{n}(\varphi):=c(\varphi)+\sum_{|\omega|<n}c_{\omega}(\varphi)h_{\omega}.
$$
Equivalently, $A_{n}$ is also the function obtained by averaging $\varphi$ on cylinders of level $n$, i.e.,
\begin{equation}\label{equ_haar trunction}
A_{n}(\varphi)=\sum_{|\omega|=n}\left(2^n\int_{[\omega]}\varphi dx\right)\chi_{[\omega]}.
\end{equation}

For a given $\varphi\in C^{0}(\Omega)$, suppose there is an integer $M$, such that
\begin{equation}\label{equ_bound conditon}
\int_{[\underbrace{0\cdots0}_{n}]}\varphi dx=\int_{[\underbrace{1\cdots1}_{n}]}\varphi dx,~~\forall n\geq M.
\end{equation}
Then $\varphi$ associates {\bf an edge weight} on quotient de Bruijn graphs $\{G_{n}\}_{n\geq M}$. That is, for each edge $e\in E_{n}$ with $e$ associating a word $\omega\in\{0,1\}^{n}$,
\begin{equation}\label{equ_graphweight}
\mbox{wgh}_{n}(e):=A_{n}(\varphi)|_{[\omega]}.
\end{equation}
In fact, Equation \eqref{equ_bound conditon} ensures the edge weight in \eqref{equ_graphweight} is well defined.

\subsection{Maximum mean cycle on quotient de Bruijn graphs and gap criterion}

We now study the maximum mean cycle problem on quotient de Bruijn graph $G$, and this is actually the combinatorial optimization truncation for our original ergodic optimization problem stated in Section \ref{sec_gelfondergodic}. To be more precise, for each cycle $\mathcal{C}\subset G$, define the mean weight $\lambda(\mathcal{C})$ of the cycle as the ratio of the sum of the weights of the cycle and the number of edges in the cycle, namely
$$
\lambda(\cC):=\frac{wgh(\cC)}{|\cC|},
$$
The maximum cycle mean $\lambda_{1}$ of $G$ is defined as
\begin{equation}\label{eq:1st_max}
\lambda_{1}:=\max_{\cC\subset G}\{\lambda(\cC)\}.
\end{equation}
The maximum mean cycle problem considers the estimation of the value of $\lambda_{1}$ and the corresponding cycle $\gamma_1$ with cycle mean $\lambda_{1}$.

Suppose $\gamma_1$ is unique, then we define the second maximum cycle mean
\begin{equation}\label{eq:2nd_max}
\lambda_{2}:=\max_{\cC\subset G,\cC\neq\gamma_1}\{\lambda(\cC)\};
\end{equation}
and let $\gamma_2$ be the corresponding cycle with cycle mean $\lambda_{2}$. The gap
\begin{equation}\label{equ_gap}
gap:=\lambda_{1}-\lambda_{2}>0.
\end{equation}
In \cite{BZ16}, the following gap condition is developed.
\begin{lemma}\cite[Lem4.1]{BZ16}\label{lem_gapcondition}
For each $\varphi\in C^{0}(\Omega)$, suppose the following gap condition is satisfied:
\begin{equation}\label{equ_gapcondition}
    \exists N\in\N,~~~\mbox{s.t.}~~~gap(A_{N}(\varphi))>\sum_{k=N}^{\infty}(k-N+1)\max_{|\omega|=k}|c_{\omega}(\varphi)|.
\end{equation}
Then the maximizing measure for $A_{N}(\varphi)$ is unique and is exactly the periodic measure $\mu$ supported on the periodic orbit associated to the cycle $\gamma_1$ with maximum cycle mean in $G_{N}$. Moreover, $\mu$ is also the unique maximizing measure for $\varphi$ and $A_{n}(\varphi),~~\forall n\geq N$.
\end{lemma}
In informal terms, the gap criterion says if the tail of the Haar series is smaller compared to the gap of its initial part, then it does not influence the maximizing measure.

\subsection{Implementation of gap criterion for Gelfond exponent}

We are ready to state our main theorem in this section. Recall that $g^{(c)}(x)=\log|\cos\pi(x+c)|$.
Fix $2\leq d\in\N$, and put
$$
g_{d}^{(c)}(x):=\max\left\{g^{(c)}(x),\log\big|\sin\frac{\pi}{2^{d}}\big|\right\}.
$$
Next, fix $2\leq d'\in\N$, without loss of generality, suppose
$$
\inf g_{d}^{(c)}|_{[0,1/2^{d'})}\geq \sup g_{d}^{(c)}|_{[1-1/2^{d'},1)}=g_{d}^{(c)}(0),
$$
and put
\begin{equation}\label{equ_target function}
g_{d,d'}^{(c)}(x):=\left\{\begin{array}{ll}
                    g_{d}^{(c)}(x), & \mbox{if}~ x\in[0,1-1/2^{d'}) \\
                    g_{d}^{(c)}(1-x), & \mbox{if}~~x\in[1-1/2^{d'},1).
                   \end{array}\right.
\end{equation}
The graph of $g_{d,d'}^{(c)}$ is pictured in Figure \ref{fig:gdd_c}, and according to equality \eqref{equ_haar trunction},\eqref{equ_bound conditon}, we construct the truncation
\begin{equation}
\label{eq:edge_w}
A_{n}(g_{d,d'}^{(c)})=\sum_{|\omega|=n}\left(2^{n}\int_{[\omega]}g_{d,d'}^{(c)}dx\right)\chi_{[\omega]},
\end{equation}
which associates edge weights on quotient de Bruijn graphs $\{G_n\}_{n\geq d'}$.
\begin{figure}
\includegraphics[width=0.6\textwidth]{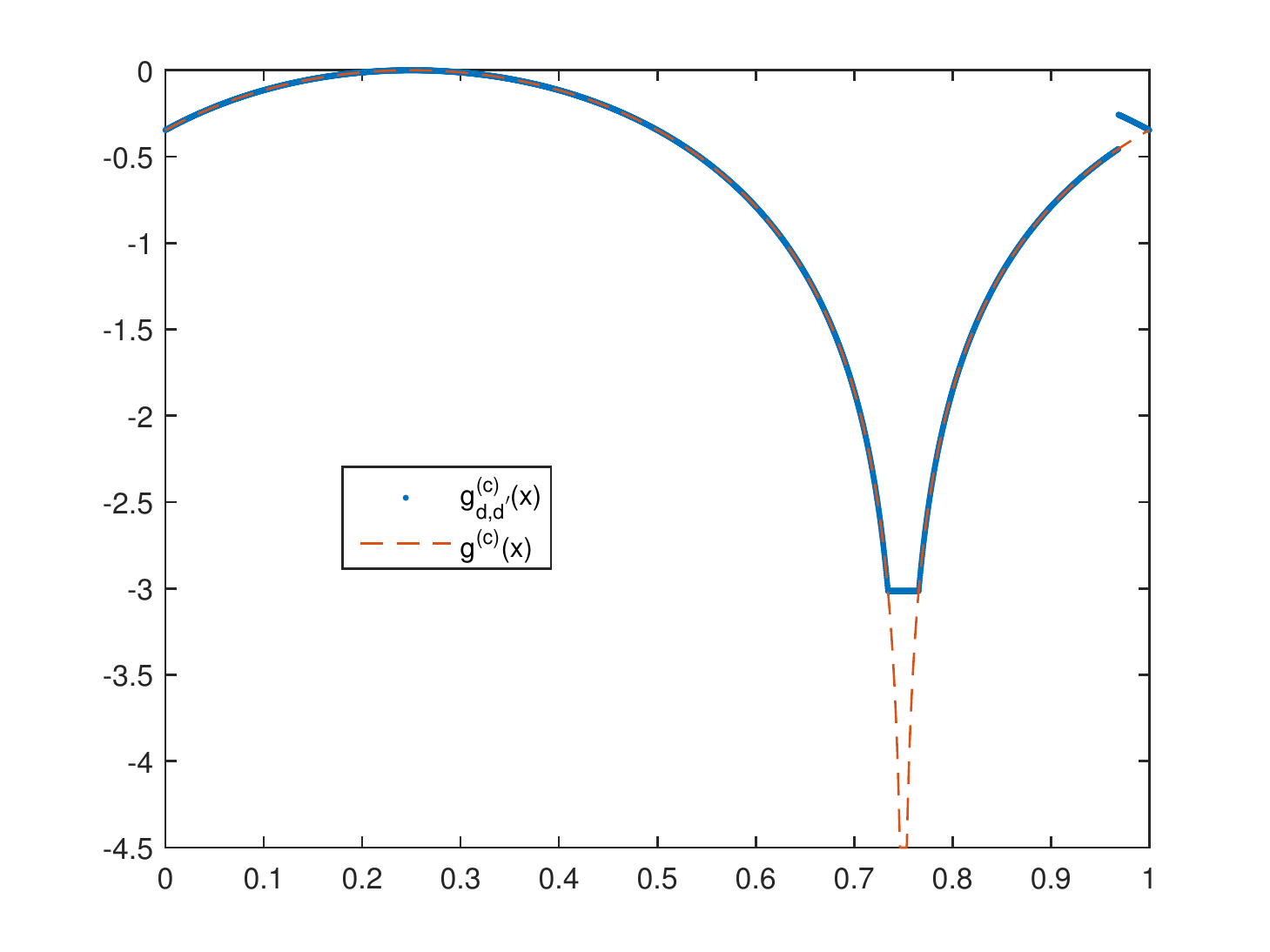}
\caption{The plot of $g_{d}^{(c)}(x) \text{ and } g_{d,d'}^{(c)}(x), \text{ with }d=6,d'=5,c=3/4$}
\label{fig:gdd_c}
\end{figure}
\begin{theorem}\label{theo_gapcondtion}
For every $c\in(0,1)$, suppose there are $d,d',N\in\N$  satisfying
\begin{equation}\label{equ_hypothesis}
\mbox{gap}\left(A_{N}(g_{d,d'}^{(c)})\right)>\frac{5\pi}{2}\cot(\frac{\pi}{2^{d}})2^{-N},
\end{equation}
then the maximizing measure for $g_{d,d'}^{(c)}$ is supported on a periodic orbit(say $\mu$). If further assuming that
\begin{equation}\label{equ_periodiccoincidence}
g_{d,d'}^{(c)}(x)=g^{(c)}(x),~~~\forall x\in \mbox{supp}\mu,
\end{equation}
then $\mu$ is also the unique maximizing measure for $g^{(c)}$.
\end{theorem}
\begin{proof}
For each $n$-bit word $\omega\in\{0,1\}^{n}$, we have
\begin{align*}
|c_{\omega}(g_{d,d'}^{(c)})|
&=\left|2^{n+2}\int g_{d,d'}^{(c)}\cdot\frac{\chi_{[\omega0]}-\chi_{[\omega1]}}{2}dx\right|\\
&=2^{n+1}\left|\int_{[\omega0]}g_{d,d'}^{(c)}dx-\int_{[\omega1]}g_{d,d'}^{(c)}dx\right|\\
&=|g_{d,d'}^{(c)}(\xi_{1})-g_{d,d'}^{(c)}(\xi_{2})|,~~~~\text{$\exists \xi_{1}\in[\omega0]$ and $\xi_{2}\in[\omega1]$}\\
&=|(g_{d,d'}^{(c)})'_{+}(\xi_{3})|\cdot|\xi_{1}-\xi_{2}|\leq |(g_{d,d'}^{(c)})'_{+}(\xi_{3})|\cdot 2^{-n},~~\text{$\exists \xi_{3}\in[\xi_{1},\xi_{2}]$}\\
&\leq \pi\cot(\frac{\pi}{2^{d}})\cdot 2^{-n}.
\end{align*}
Therefore,
\begin{align*}
\sum_{k=n}^{\infty}(k-n+1)\max_{|\omega|=k}|c_{\omega}(g_{d,d'}^{(c)})|
&\leq \sum_{k=n}^{\infty}(k-n+1)\pi\cot(\frac{\pi}{2^d})2^{-k}\\
&=\pi\cot(\frac{\pi}{2^d})2^{-n}+\sum_{k=n+1}^{\infty}(k-n+1)\pi\cot(\frac{\pi}{2^{d}})2^{-k}\\
&=\frac{5\pi}{2}\cot(\frac{\pi}{2^d})2^{-n}.
\end{align*}
Together with hypothesis \eqref{equ_hypothesis}, there exists $N\in\N$, such that
$$
\mbox{gap}(A_{N}(g_{d,d'}^{(c)}))>\frac{5\pi}{2}\cot(\frac{\pi}{2^d})\cdot2^{-N}\geq\sum_{k=N}^{\infty}(k-N+1)\max_{|\omega|=k}|c_{\omega}(g_{d,d'}^{(c)})|.
$$
So the gap criterion (Lemma \ref{lem_gapcondition}) directly yields that the maximizing measure for $g_{d,d'}^{(c)}$ is unique and periodic (say $\mu$), and $\mu$ can be determined by the truncation $A_{N}(g_{d,d'}^{(c)})$. This completes the first assertion of the theorem.

Next note that $g_{d,d'}^{(c)}(x)\geq g^{(c)}(x),~~\forall d,d'\in\N~~\mbox{and}~~x\in[0,1)$, therefore, the second assertion of the theorem directly follows from Proposition \ref{lem_monotonemaximizing}, and thus the proof of Theorem \ref{theo_gapcondtion} is complete.
%
\end{proof}

\section{Verification of the assumption in Theorem \ref{theo_gapcondtion}}\label{Sec_verification}

\subsection{The algorithm for verifying the gap criterion}

Needless to say, Theorem 3.2 is meaningless unless the assumption admits a non-empty subset of $c\in (0,1)$ such that there exists some $(d,d',N)\in\mathbb{R}^3$ for which \eqref{equ_hypothesis} and the statements that follow all hold true (hence the gap criterion holds). Since $(0,1)$ is uncountable, we only enumerate over the subset $\Lambda_{L}=\left\{ \frac{i}{2^{L}}:i=0,\ldots,2^{L}-1\right\}$ for a fixed integer $L\in\mathbb{N}$, and check over the finite set of integer triples  $T_{d,d',N}\subset\{(d,d',N)\in\mathbb{R}^3:d\leq d'\leq N\}$
to see if
\begin{equation}\label{eq:gap_inequality}
\mbox{gap}(A_{N}(g_{d,d'}^{(c)}))>2^{-N}\frac{5}{2}\pi\cot\left(\frac{\pi}{2^{d}}\right).
\end{equation}
If so, then the maximizing measure $g^{(c)}$ exists for that $c\in \Lambda_{L} $. The enumeration process is terminated once a triple $(d,d',N)$ is found for a given $c$ such that the assumptions in Theorem \ref{theo_gapcondtion} hold, or $N$ is larger than some upper limit $N_{\max}$. The total number of cases checked for each $c$ is less than $(N_{\max})^3$.

The left-hand side of (\ref{eq:gap_inequality}) represents the gap between the maximum cycle mean ($\mbox{MCM}_1$) and second maximum cycle mean ($\mbox{MCM}_2$) of $N$-th quotient De Bruijn graph ${G}_N=({V}_N,{E}_N)$ associating with edge weights $\mbox{wgh}_N={A}_{N}(g_{d,d'}^{(c)})$, where the formula for $A_{N}(g_{d,d'}^{(c)})$ is given in \eqref{equ_graphweight}. The cycle with maximum mean weight of the given graph is found by a so-called Howard's algorithm as described in \cite{DRG,Howard}, and also formulated in Algorithms in the Appendix \ref{sec_appendix}.

Howard's algorithm gets input $G_N$ and outputs the cycle $\Gamma(A_N(g_{d,d'}^{(c)}))$ with maximum mean weight $\text{MCM}_1(A_N(g_{d,d'}^{(c)}))$. The algorithm for finding the second maximum cycle mean $\text{MCM}_2(A_N(g_{d,d'}^{(c)}))$ is based on that for $\mbox{MCM}_1(A_N(g_{d,d'}^{(c)}))$. More precisely,
\begin{equation}
\mbox{MCM}_2(A_N(g_{d,d'}^{(c)})) = \underset{e\in \Gamma(A_N(g_{d,d'}^{(c)}))}{\max} \mbox{MCM}_1(A_{N,\backslash e}(g_{d,d'}^{(c)}))
\end{equation}
where $A_{N,\backslash e}(g_{d,d'}^{(c)})$ is the weight of the graph obtained by removing the edge $e$ (along with its weight) from the graph $G_N$ while keeping all vertices. By definition,
\begin{equation}
\text{gap}(A_{N}(g_{d,d'}^{(c)})) = \text{MCM}_1(A_N(g_{d,d'}^{(c)})) - \text{MCM}_2(A_N(g_{d,d'}^{(c)})).
\end{equation}
We introduce function $\psi:\Lambda_{L}\rightarrow \mathbb{N}$ such that
\begin{equation}
\psi(c)=\left\{\begin{array}{ll}
|\Gamma(A_N(g_{d,d'}^{(c)}))| &  \exists (d,d',N)\in T_{d,d',N}  \text{ s.t. }\eqref{eq:gap_inequality}, \\
0 & \text{ otherwise.}
\end{array}\right.
\end{equation}
The complete algorithm is outlined in Algorithm \ref{alg:gap_criterion}.
\begin{algorithm}[h]
\caption{Verification of the gap criterion}
\label{alg:gap_criterion}
\begin{algorithmic}[1]	
\Require $\Lambda_{L}, T_{d,d',N}$
\Ensure $\psi(c)$ for $c\in \Lambda_{L}$
\For {every $c\in \Lambda_{L}$}
\State Set $\psi(c)=0$
\For {$d=d_{\min}$ to $d_{\max}$}
\For {$d'=d$ to $d'_{\max}$}
\For {$N=d'$ to $N_{\max}$}
\State Compute edge weights $A_{N}(g_{d,d'}^{(c)})$ for $G_N$.
\State Set $\text{gap}(A_{N}(g_{d,d'}^{(c)})) =0$
\State Calculate $\text{MCM}_1(A_N(g_{d,d'}^{(c)}))$ and $\Gamma(A_N(g_{d,d'}^{(c)}))$ by Algorithm \ref{alg:policy_iteration}.
\For {every edge $e$ in $\Gamma(A_N(g_{d,d'}^{(c)}))$}
\State Calculate $\mbox{MCM}_1(A_{N,\backslash e}(g_{d,d'}^{(c)}))$ by Algorithm \ref{alg:policy_iteration}.
\State Set $\text{gap}(A_{N}(g_{d,d'}^{(c)})) = \text{MCM}_1(A_N(g_{d,d'}^{(c)})) -\mbox{MCM}_1(A_{N,\backslash e}(g_{d,d'}^{(c)}))$
\If {$ \text{gap}(A_{N}(g_{d,d'}^{(c)}))> 2^{-N}\frac{5}{2}\pi\cot\left(\frac{\pi}{2^{d}}\right) $}
\State Set $\psi(c)=|\Gamma(A_N(g_{d,d'}^{(c)}))|$.
\State Terminate the for-loops of $d,d',N$.
\EndIf
\EndFor
\EndFor
\EndFor
\EndFor
\EndFor
\State \Return $\psi(c)$.
\end{algorithmic}
\end{algorithm}

\bigskip

\subsection{Numerical implementation and error estimate}

Numerical errors arise in the execution of Algorithm \ref{alg:gap_criterion}. Enumeration and counting of integers are generally considered error-free. The numerical errors come from machine error, and the numerical approximation error in particular in numerical integrations.
The right-hand side of \eqref{eq:gap_inequality} can be evaluated within
relative error $\epsilon_m$,  where $\epsilon_m$ is machine epsilon. More exactly, the numerical error of evaluating the right-hand side of (\ref{eq:gap_inequality}) is bounded by
\begin{equation}
2^{-N}\frac{5}{2}\pi|\cot(\frac{\pi}{2^{d}})|\cdot \epsilon_m.
\end{equation} In the following, we discuss the numerical implementation and error estimate
for the left-hand side and analyse both the forward and backward numerical
errors.

First, we compute $A_{N}(g_{d,d'}^{(c)})$ by numerical integration
of $\int_{[\omega]}g_{d,d'}^{(c)}(x)\,\mathrm{d}x$. Here we use the rectangle method, where the interval $[\omega]=[x_0,x_{2^{26-N}}]$ is divided into subintervals of length $\Delta x=2^{-26}$, and are represented in turn as
 $[x_0,x_1], \ldots, [x_{2^{26-N}-1},x_{2^{26-N}}]$. By symmetry of  $g_{d,d'}^{(c)}(x)$ on $[\omega]$, we only need to calculate the integral on  $[x_0, x_{2^{25-N}}]$. The numerical integration of $\int_{[x_0, x_{2^{25-N}}]}g_{d,d'}^{(c)}(x)\,\mathrm{d}x$  by rectangle method has two forms $I_L, I_R$ which are written as
\begin{equation}
I_L=\sum_{j=0}^{x_{2^{25-N}-1}}g_{d,d'}^{(c)}(x_j)\Delta x.
\end{equation} and
\begin{equation}
I_R=\sum_{j=1}^{x_{2^{25-N}}}g_{d,d'}^{(c)}(x_j)\Delta x.
\end{equation}
Since $g_{d,d'}^{(c)}(x)$ is  monotone on $[\omega]$, we have the error estimate
\begin{equation}
\max\left\{\left|I_R- \int_{[\omega]}g_{d,d'}^{(c)}(x) \right|, \left|I_L- \int_{[\omega]}g_{d,d'}^{(c)}(x) \right|\right\}\leq |I_L-I_R|.
\end{equation}where
\begin{equation}
|I_L-I_R| =\left|g_{d,d'}^{(c)}(x_0)-g_{d,d'}^{(c)}(x_{2^{25-N}})\right|\Delta x.
\end{equation}
Taking into account the machine error,  the numerical error for $\int_{[\omega]}g_{d,d'}^{(c)}(x)\,\mathrm{d}x$ is bounded by
\begin{equation} \label{eq:error_epsilon}
\epsilon = \left|g_{d,d'}^{(c)}(x_0)-g_{d,d'}^{(c)}(x_{2^{25-N}})\right|\Delta x(1+\epsilon_m),
\end{equation}
which is uniform for all $[\omega]$. Therefore,
\begin{equation}
\|\tilde{A}_{N}(g_{d,d'}^{(c)}) - A_{N}(g_{d,d'}^{(c)}))\|_{\infty}=\epsilon.
\end{equation}

The left-hand side of (\ref{eq:gap_inequality}) represents the gap between the maximum cycle mean ($\mbox{MCM}_1$) and second maximum cycle mean ($\mbox{MCM}_2$) of $N$-th De Bruijn graph with weights $\tilde{A}_{N}(g_{d,d'}^{(c)})$. Here $\tilde{A}_{N}(g_{d,d'}^{(c)})$ is the numerical value of ${A}_{N}(g_{d,d'}^{(c)})$, which has the error bounded by $\epsilon$ as defined in Eq.(\ref{eq:error_epsilon}). By definition,
\begin{equation}
\text{gap}(A_{N}(g_{d,d'}^{(c)})) = \text{MCM}_1({A}_{N}(g_{d,d'}^{(c)})) - \text{MCM}_2({A}_{N}(g_{d,d'}^{(c)})),
\end{equation}
and its numerical value is given by
\begin{equation}
\text{gap}(\tilde{A}_{N}(g_{d,d'}^{(c)})) = \text{MCM}_1(\tilde{A}_{N}(g_{d,d'}^{(c)}))- \text{MCM}_2(\tilde{A}_{N}(g_{d,d'}^{(c)})).
\end{equation}
The forward error of the numerical algorithm for the function $ \text{MCM}_1:\mathbb{R}_+^m \rightarrow \mathbb{R}_+$ comes from evaluation of the cycle mean, which is accurate up to relative error bounded by machine epsilon $\epsilon_m$.  Also we show that the algorithm is  backward stable.
It is obvious that if the weight of each edge of a graph is increased by $\epsilon$, the cycle with maximum mean is unchanged. Therefore,
\begin{equation}
\text{MCM}_1({A}_{N}(g_{d,d'}^{(c)})) + \epsilon = \text{MCM}_1({A}_{N}(g_{d,d'}^{(c)})+\epsilon).
\end{equation}
For $G_N$ with edge weight $\tilde{A}_{N}(g_{d,d'}^{(c)})$, we denote the cycle with maximum mean weight by $\tilde\cC$. For $G_N$ with edge weight ${A}_{N}(g_{d,d'}^{(c)})+\epsilon$, the mean weight of $\tilde\cC$ is less than or equal to $\text{MCM}_1({A}_{N}(g_{d,d'}^{(c)})+\epsilon)$, but greater than or equal to $\text{MCM}_1(\tilde{A}_{N}(g_{d,d'}^{(c)}))$. Therefore, we have
\begin{equation}
\text{MCM}_1(\tilde{A}_{N}(g_{d,d'}^{(c)})) \leq \text{MCM}_1({A}_{N}(g_{d,d'}^{(c)})) + \epsilon.
\end{equation}
By similar arguments, we can show that
\begin{equation}
\text{MCM}_1(\tilde{A}_{N}(g_{d,d'}^{(c)})) \geq \text{MCM}_1({A}_{N}(g_{d,d'}^{(c)})) - \epsilon.
\end{equation}
 It is readily checked that
\begin{equation}\label{eq:error_bound1}
|\text{MCM}_1(\tilde{A}_{N}(g_{d,d'}^{(c)})) - \text{MCM}_1({A}_{N}(g_{d,d'}^{(c)}))| \leq \|\tilde{A}_{N}(g_{d,d'}^{(c)}) - A_{N}(g_{d,d'}^{(c)}))\|_{\infty}=\epsilon,
\end{equation}
where $\epsilon$ is computed in \eqref{eq:error_epsilon}.
Similarly we have
\begin{equation}
|\text{MCM}_2(\tilde{A}_{N}(g_{d,d'}^{(c)})) - \text{MCM}_2({A}_{N}(g_{d,d'}^{(c)}))| \leq \epsilon.
\end{equation}

Based on the above backward error analysis, a sufficient condition for inequality (\ref{eq:gap_inequality}) is
\begin{equation}
\text{gap}(\tilde{A}_{N}(g_{d,d'}^{(c)})) \geq 2^{-N}\frac{5}{2}\pi\cot(\frac{\pi}{2^{d}}) + 2\epsilon
\end{equation}
for $\epsilon$ defined in Eq.(\ref{eq:error_epsilon}). If the forward error is taken into account, then the sufficient condition becomes
\begin{equation}\label{eq:_numer_gap_inequality}
\text{gap}(\tilde{A}_{N}(g_{d,d'}^{(c)}))> 2^{-N}\frac{5}{2}\pi\cot\left(\frac{\pi}{2^{d}}\right) + 2\epsilon + \left(2^{-N}\frac{5}{2}\pi\cot\left(\frac{\pi}{2^{d}}\right) + \text{gap}(\tilde{A}_{N}(g_{d,d'}^{(c)}))\right)\cdot\epsilon_m,
\end{equation}
for the machine epsilon $\epsilon_m$. The last term represents the forward error for numerical computation, which is negligible in this problem.


\section{Outputs}\label{sec_output}
In this section, we provide a detailed explanation of the mechanism for drawing Figure \ref{fig:butterfly}, as well as the robustness results stated in Section \ref{sec_introduction}.
\subsection{Drawing Figure \ref{fig:butterfly}}
We choose
$$
\Lambda_{10}=\left\{ \frac{i}{2^{10}}:i=0,\ldots,2^{10}-1\right\},
$$
and
$$
T_{d,d',N}=\left\{(d,d',N):d=3,\ldots,15;d'=3,\ldots,22,N=3,\ldots,22\right\}.
$$
For each $c\in \Lambda_{10}$, we enumerate the set of triples for the parameters $T_{d,d',N}$
and verify Inequality (\ref{eq:gap_inequality}), by checking the Inequality (\ref{eq:_numer_gap_inequality}). In particular, we have
\begin{example}\label{example:classicalthuemorse}
  When specializing $c=1/2$, it follows from our Algorithm \ref{alg:gap_criterion}, that inequality \eqref{eq:_numer_gap_inequality} holds at $d=3,d'=3,N=10$. Moreover, Algorithm \ref{alg:gap_criterion} also finds the simple cycle coded with $\overline{01}$ in $10$-th quotient de Bruijn graph $G_{10}$ is the unique cycle with maximizing mean weight. Therefore, following from Theorem \ref{theo_gapcondtion}, we have the invariant periodic measure $\frac{1}{2}(\delta_{1/3}+\delta_{2/3})$ as the unique maximizing measure for $g_{1/2}$, and thus
  $$
  \Delta^{(1/2)}=1+\left(\int g^{(1/2)}d\frac{1}{2}(\delta_{1/3}+\delta_{2/3})\right)/\log2=\frac{\log3}{\log4}\approx0.79248125.
  $$
\end{example}

We show that there are some other examples analogous to Example \ref{example:classicalthuemorse}.
\begin{example}\label{example:positivecases}
  When specializing $c=1/4$ or $3/8$, it follows from our Algorithm \ref{alg:gap_criterion}, that inequality \eqref{eq:_numer_gap_inequality} holds at $d=3,d'=3,N=15$. Moreover, Algorithm \ref{alg:gap_criterion} also finds that the cycle coded with $\overline{0111}$ and $\overline{011}$ in the $15$-th quotient de Bruijn graph $G_{15}$ are the unique cycle with maximizing mean weight respectively. Therefore, the periodic invariant measures $\frac{1}{4}(\delta_{7/15}+\delta_{14/15}+\delta_{13/15}+\delta_{11/15})$ and $\frac{1}{3}(\delta_{3/7}+\delta_{6/7}+\delta_{5/7})$ are the unique maximizing measure for $g^{(1/4)}$ and $g^{(3/8)}$ respectively. Thus
  $$
  \Delta^{(1/4)}=1+\left(\int g^{(1/4)}d\frac{1}{4}(\delta_{7/15}+\delta_{14/15}+\delta_{13/15}+\delta_{11/15})\right)/\log2\approx0.7442276;
  $$
  and
  $$
  \Delta^{(3/8)}=1+\left(\int g^{(3/8)}d\frac{1}{3}(\delta_{3/7}+\delta_{6/7}+\delta_{5/7})\right)/\log2\approx 0.7416347.
  $$
\end{example}

However, there are examples in contrast to Example \ref{example:classicalthuemorse} and \ref{example:positivecases}.
\begin{example}\label{example:negativecases}
When specializing $c=\frac{192}{1024}$, for all the possible choices in $T_{d,d',N}$, inequality \eqref{eq:_numer_gap_inequality} never holds. Therefore, we can't deduce from our Algorithm \ref{alg:gap_criterion} up to $T_{d,d',N}$ level, about the maximizing measure for $g^{(\frac{192}{1024})}$. Accordingly, $\Delta^{(\frac{192}{1024})}$ is untestable from our method, up to $T_{d,d',N}$ level.
\end{example}

The comprehensive enumeration data for every $c\in\Lambda_{10}$ are available at the site:
\url{https://www.dropbox.com/s/y5j0ez3v96ld7a3/formalized_result521%281%29.xlsx?dl=0}, and these data form the Gelfond exponent function (i.e., the blue dot curve) in Figure \ref{fig:butterfly}.

From the enumeration, for each $c\in\Lambda_{10}\backslash\Lambda_{10}^{U}$, there is a triple $(d,d',N)\in T_{d,d',N}$, such that inequality \eqref{eq:_numer_gap_inequality} holds in Algorithm \ref{alg:gap_criterion}. Thus, from Theorem \ref{theo_gapcondtion}, there is always a unique periodic invariant measure $\mu_{c}$, which is the maximizing measure for $g^{(c)}$, and accordingly,
\begin{equation}\label{equ_Gelfondexponentgood}
 \Delta^{(c)}=1+\frac{\int g^{(c)}d\mu_{c}}{\log2}=1+\frac{\sum_{p_{i}\in supp(\mu_{c}),i=1,\cdots,per(\mu_{c})}g^{(c)}(p_{i})}{\log2}.
\end{equation}
One the other side, for each $c\in\Lambda_{10}^{U}$, our Algorithm \ref{alg:gap_criterion} is untestable up to $T_{d,d',N}$ level.

\medskip

Next, we prove the real analyticity of the Gelfond exponent function $c\to\Delta^{(c)}$ on every $c\in\Lambda_{10}\backslash\Lambda_{10}^{U}$. This can be deduced as follows. First, fix any $c\in\Lambda_{10}\backslash\Lambda^{U}_{10}$, we know there is a triple $(d,d',N)\in T_{d,d',N}$ such that Inequality \ref{eq:_numer_gap_inequality} holds.
By the continuity of $\mbox{gap}$ function, it is clear that there exists an open neighborhood $U$ centered at $c$, such that for every $c'\in U$, we still have
$$
\text{gap}(\tilde{A}_{N}(g_{d,d'}^{(c')})) \geq 2^{-N}\frac{5}{2}\pi\cot\left(\frac{\pi}{2^{d}}\right) + 2\epsilon + \left(2^{-N}\frac{5}{2}\pi\cot\left(\frac{\pi}{2^{d}}\right) + \text{gap}(\tilde{A}_{N}(g_{d,d'}^{(c')}))\right)\cdot\epsilon_m,
$$
and $\mu_{c}$ is also the cycle with maximizing mean weight in $G_{N}$. By applying Theorem \ref{theo_gapcondtion} for parameter $c'$, we have $\mu_{c}$ is the universal periodic invariant maximizing measure for $g^{(c')}$. Moreover,
\begin{eqnarray}
  |\Delta^{(c)}-\Delta^{(c')}| &=& \frac{1}{\log2}\sum_{p_{i}\in supp(\mu_{c}),i=1,\cdots,per(\mu_{c})}g^{(c)}(p_{i})|g^{(c)}(p_{i})-g^{(c')}(p_{i})|,
\end{eqnarray}
and every $p_{i}$ is away from the logarithm pole of $g^{(c)}$ and $g^{(c')}$. At each $p_{i}$, $g^{(c)}(p_{i})$ is real analytic with respect to parameter $c$, thus, we have $\Delta^{(c)}$ is real analytic at every $c\in\Lambda_{10}\backslash\Lambda^{U}_{10}$.

\medskip

\subsection{Hunt-Ott type's results}

Based on our enumeration data, we are able to do some Hunt-Ott type's results (as what they have done for the trigonometric potential functions in \cite{HO96,HO962}), other than Figure \ref{fig:butterfly}. We believe these results will have their independent interest.

Let $h: \Lambda_{10}\rightarrow \mathbb{N}$, defined by
\begin{equation}\label{equ_lengthfunction}
h(c)=\left\{\begin{array}{ll}
\mbox{Per}(\mu_{c}) &  c\in \Lambda_{10}\backslash\Lambda_{10}^{U}, \\
0 & \text{ otherwise.}
\end{array}\right.
\end{equation}
We plot the graph of $h$ at Figure \ref{fig:cat}.

\begin{figure}
\includegraphics[width=0.8\textwidth]{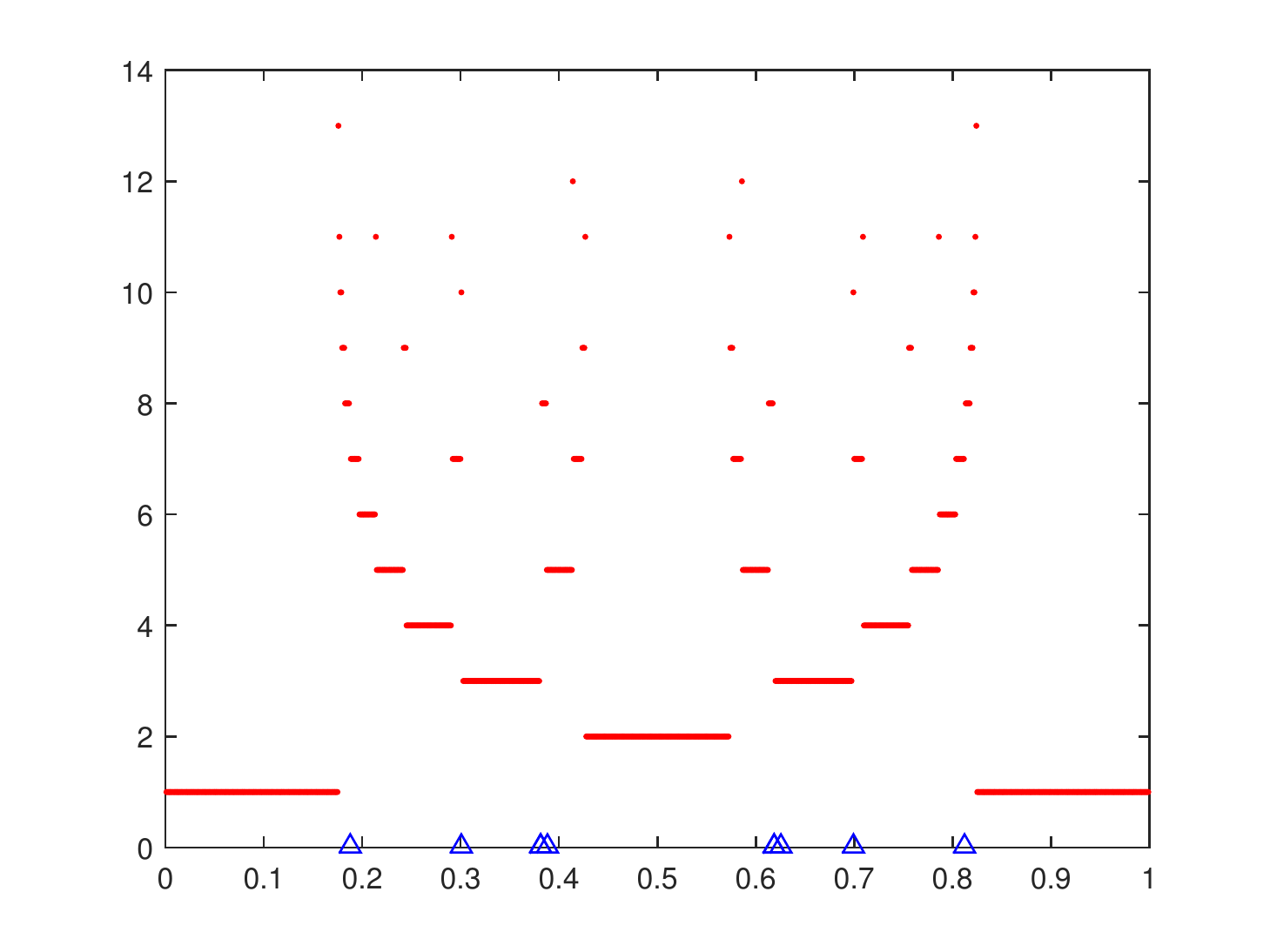}
\caption{Graph of function $h$ in Equation \eqref{equ_lengthfunction}. The 8 triangles are the values $c$ where our Algorithm \ref{alg:gap_criterion} untestable, and these 8 points seem like the ``tipping points'' for sudden changes of the period of the maximizing measures.}
\label{fig:cat}
\end{figure}

Next, denote by the $\Gamma$-interval $[a,b]$, the largest connected interval (which allows $a=b$, such that the interval shrinks to be a singleton) for $a,b\in \Lambda_{10}$ with $\mbox{Per}(\mu_{c})\equiv p_{\Gamma}~~(\mbox{a constant}),~~\forall a\leq c\leq b$. Denote also the $\Gamma$-interval for each $c\in\Lambda_{10}\backslash\Lambda_{10}^{U}$ by a singleton at $c$. It is observed a Farey tree type-structure appearing in Figure \ref{fig:cat}. That is, between any two $\Gamma$-intervals with optimal periodic measure of period $p_{1}$ and $p_{2}$,  there always exists a smaller $\Gamma$-interval with optimal periodic measure of period $p_{1}+p_{2}$, and all the other $\Gamma$-intervals in between have period greater than $p_{1}+p_{2}$. For example, consider the subinterval $[\frac{155}{512},\frac{293}{512}]$. Between the period 2-interval $[\frac{219}{512},\frac{293}{512}]$ and 3-interval $[\frac{155}{512},\frac{389}{1024}]$, there is a 5-interval $[\frac{397}{1024},\frac{423}{1024}]$, while between the formal 3-interval and 5-interval, there is a 8-interval $[\frac{49}{128},\frac{99}{256}]$.

On the other hand, let
$$
\rho(p):=\frac{1}{1024}\sharp\{c\in\Lambda_{10},~~\mbox{and}~~Per(\mu_{c})=p\}.
$$
Figure \ref{fig:cycle_dist} illustrates the distribution of $\rho(p)$, and it is clearly observed that the $\rho(p)$ asymptotically decreases, as $p$ increases. Moreover, Figure \ref{fig:linear_fit} fits
\begin{equation}\label{equ_asymptoticbehavior}
    \rho(p)\backsim p\cdot(\exp(-0.78))^{p}\Phi(p),
\end{equation}
where $\Phi(p)$ is the Euler function.
\begin{figure}
\includegraphics[width=0.6\textwidth]{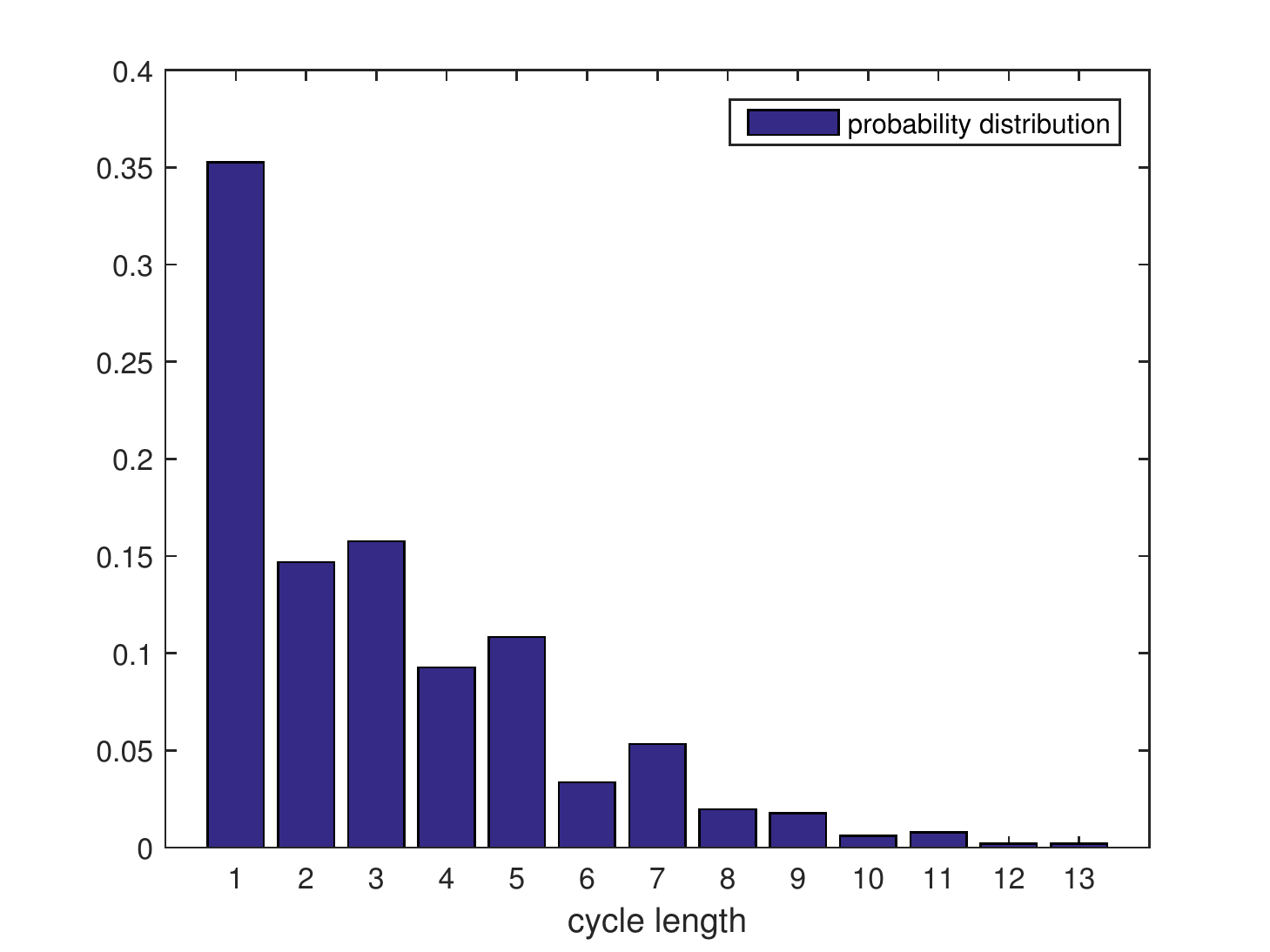}
\caption{$\rho(p)$ vs. $p$.}
\label{fig:cycle_dist}
\end{figure}

\begin{figure}
\includegraphics[width=0.6\textwidth]{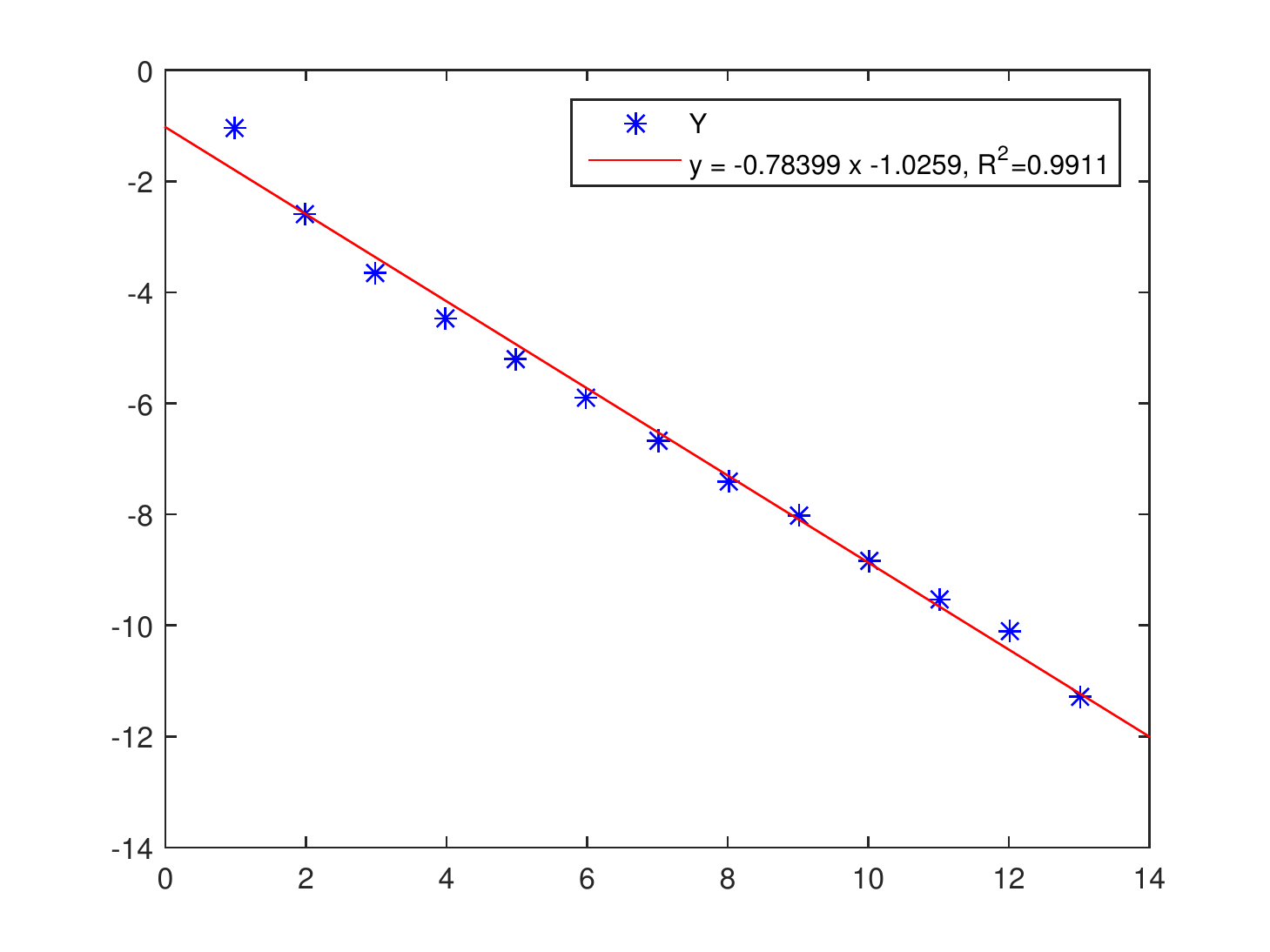}
\caption{Fitting of the formula for distribution as $\log(\rho(p)/p\Phi(p))$ vs $p$. The R-square is 0.9911, indicating a strong goodness-of-fit.}
\label{fig:linear_fit}
\end{figure}

Finally, we zoom in to the neighborhood of $c=390/1024,391/1024$, where our Algorithm \ref{alg:gap_criterion} is untestable. As shown in Figure \ref{fig:zoomin}, those untestable values seem to be isolated (in a refined scaling).

\begin{figure}
\includegraphics[width=0.6\textwidth]{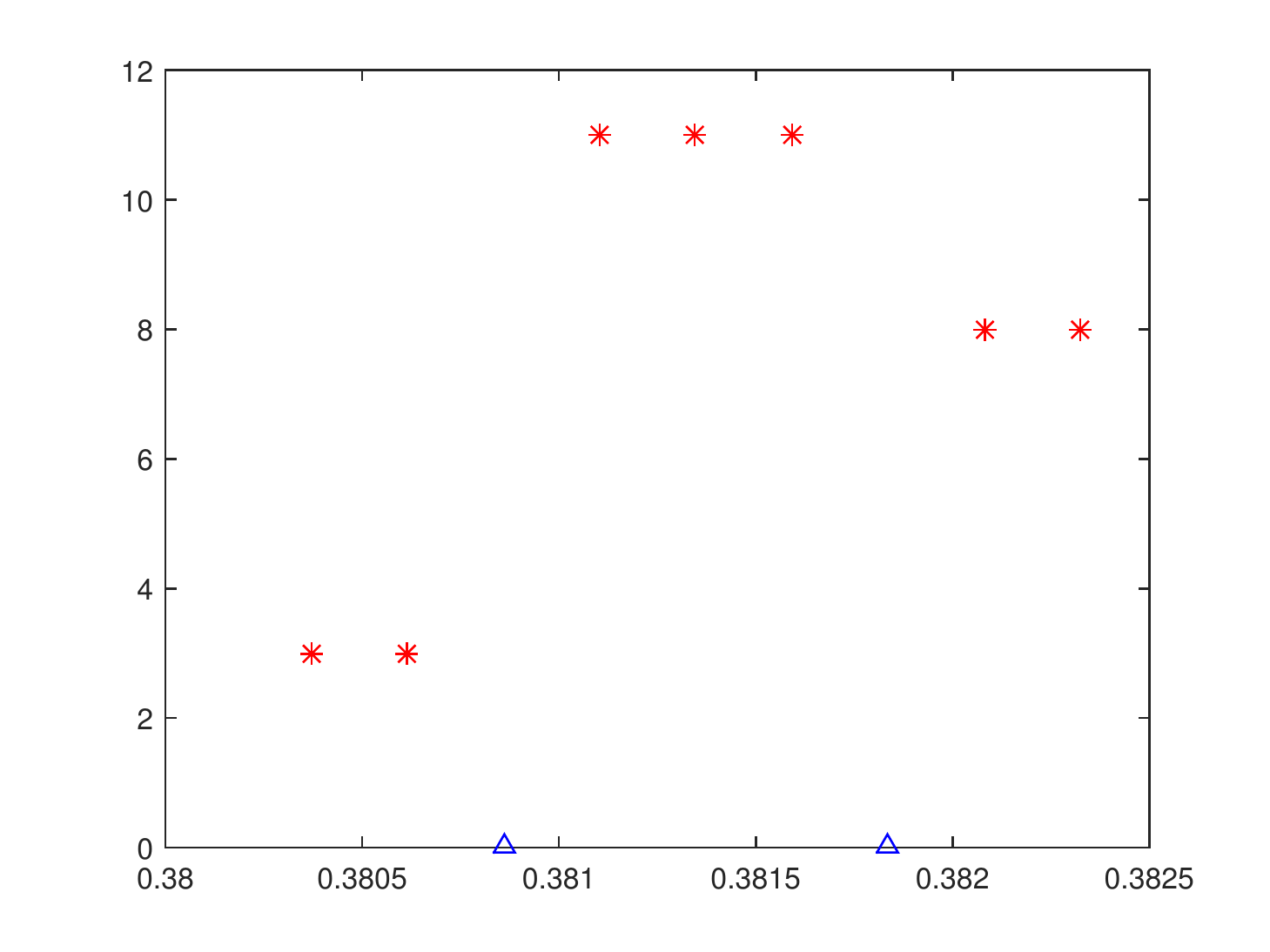}
\caption{We zoom in Figure \ref{fig:cat} at the subinterval $[0.38,0.3815]$, by computing points in $\Lambda_{12}$ (shown in star dots). Every point is testable by the Algorithm \ref{alg:gap_criterion}. Thus the untestable points $\frac{390}{1024},\frac{391}{1024}$ are isolated among $\Lambda_{12}$.}
\label{fig:zoomin}
\end{figure}


\section{Appendix: Howard's algorithm for finding the maximum cycle mean}\label{sec_appendix}

In this Appendix, we provide some detailed descriptions of Howard's algorithm.

\subsection{Max-plus algebra and the spectral problem}

The max-plus algebra is defined on the max-plus semiring $\mathbb{R}_{\text{max}}=\mathbb{R}\cup \{-\infty\}$. The addition operation $\oplus$ is defined as the binary max and the multiplication operation $\otimes$ is the usual addition $+$ on $\mathbb{R}$. More exactly, $a\oplus b = \max\{a,b\}$, and $a\otimes b= a+b$ for all $a,b\in \mathbb{R}$, with special cases $a\oplus\mathbb{O} = a$, $a\otimes \mathbf{1}=a$ and $a\otimes \mathbb{O}=\mathbb{O}$. The zero element $\mathbb{O}=-\infty$ and the unit element $\mathbf{1}=0$ are defined as such. The operations on $\mathbb{R}_{\max}$ are commutative and associative as those on $\mathbb{R}$.

Given a matrix $A\in (\mathbb{R}_{\max})^{n\times n}$, the spectral problem is written as
\begin{equation}\label{eq:spectral}
Ax=\lambda x,
\end{equation}
where $x\in (\mathbb{R}_{\text{max}})^n\backslash \{\mathbb{O}^n\}$ and $\lambda\in \mathbb{R}_{\text{max}}$.

The distinction about the spectral problem in max-plus algebra is that if the eigenmode exists, the algorithm for finding it can always be terminated in finite iterations and the solution is exact. In contrast, in usual algebra the eigenvalue problem always requires iterative methods that can only approximately calculate the eigenvalue, with the approximation error depending on the condition number of the matrix and the number of iterations taken.

\subsection{Howard's algorithm for finding the maximum cycle mean}

Given a directed graph $G=(V,E)$, equipped with edge weight $w:E\rightarrow \mathbb{R}_+$. The maximum cycle mean for a strongly connected graph $G$ is given by the (unique) eigenvalue of the spectral problem \eqref{eq:spectral} in the max-plus algebra, where $A$ is the adjoint matrix for $G$, i.e., $A_{i,j}=w(i,j)\;{\text{if }}(i,j)\in E$ and $A_{i,j}=\mathbb{O}$ otherwise. The matrix $A$ is called irreducible if $G$ is strongly connected. The uniqueness of the eigenvalue for an irreducible square matrix in the max-plus algebra is shown in \cite{Howard}. Moreover, the eigenvalue can be found exactly in finitely many steps \cite{Howard}.

The well received method to solve \eqref{eq:spectral} is based on Karp's algorithm which has time complexity $O(n^3)$ and space complexity $O(n)$. Cochet-Terrassion et. al. in \cite{Howard} propose another finite-step termination algorithm with almost linear average time complexity. It is based on the specialization of Howard's policy improvement scheme to max-plus algebra. We use this method in our computation. In the following, we briefly describe Howard's policy improvement scheme, with adaptation to solve the spectral problem in max-plus algebra. The description would generally follow from \cite{Howard}.

It is easily checked that \eqref{eq:spectral} is equivalent to
\begin{equation}
\underset{1\leq j \leq n}{\max}(A_{ij}+x_j)=\lambda + x_i,\;\forall i\in\{1,\ldots,n\}.
\end{equation}
For simplicity of arguments, we use $G$ and its corresponding adjoint matrix $A$ interchangeably. In what follows, we assume that $A$ is irreducible (or equivalently $G$ is strongly connected). For each edge $(i,j)\in E$, we denote initial node map $\text{In}(i,j)=i$, and terminal node map $\text{Out}(i,j)=j$. The policy is a map
\begin{equation}
\pi: V\rightarrow E,
\end{equation}
such that $\text{In}(\pi(i))=i,\;\forall i\in E$. The matrix $A^{\pi}$ associated with the policy $\pi$ is defined as
\begin{equation}
A^{\pi}_{ij}=\left\{\begin{array}{ll}
w(\pi(i)) & \text{if }j=\text{Out}(\pi(i))\\
\mathbb{O} &\text{otherwise}
\end{array}
\right.
\end{equation}
The algorithm for finding $\lambda$ in \eqref{eq:spectral} is summarized in Algorithm \ref{alg:policy_iteration}, which requires two subroutines \ref{alg:value_determin} and \ref{alg:policy_improve}.
First, Algorithm \ref{alg:value_determin} finds the eigenvalue-eigenvector pair (also called eigenmode) $(\eta,x)$ of matrix $A^{\pi}$.
\begin{algorithm}[h]
\caption{Value Determination}
\label{alg:value_determin}
\begin{algorithmic}[1]	
\Require $A,\pi$
\Ensure $\eta=(\eta_i,\ldots,\eta_n),\;x=(x_1,\ldots,x_n)$
\State Find a cycle $c$ in $A^{\pi}$ \label{value_begin}
\State Calculate mean cycle weight $\bar{\eta}=\frac{\sum_{e\in c}w(e)}{\sum_{e\in c}1}$.
\State Select an arbitrary node $i\in c$
\State Set $\eta_i=\bar{\eta},\; x_i=0$.
\For {all the nodes $j$ that have access to $i$ in backward topological order}
\State Set $\eta_j=\bar{\eta},\; x_j=w(\pi(j))-\bar{\eta}+x_{\text{Out}(\pi(j))}$
\EndFor \label{value_end}
\While {there exists a nonempty set $C$ not having access to $i$}
\State Repeat steps \ref{value_begin} to \ref{value_end} using the restriction of $A$ and $\pi$ to $C$ in place of $A$ and $\pi$.
\EndWhile
\State \Return
\end{algorithmic}
\end{algorithm}

Second, given a policy $\pi$, together with an eigenmode $(\eta,x)$ of $A^{\pi}$, Algorithm \ref{alg:policy_improve} finds a ``better'' policy $\pi'$. By better we mean that $\chi(A^{\pi'})\geq \chi(A^{\pi})$, where cycle time vector $\chi(A)$ is defined as
\begin{equation}
\chi(A)=\underset{k\rightarrow\infty}{\lim}\frac{1}{k} \times A^{k}x.
\end{equation}
Here we note that $\chi(A)$ exists and is independent of $x\in\mathbb{R}^n$. The proof can be found in \cite{Howard}, but the intuition is comparing it to the power iteration for finding the maximal norm eigenvalue of a symmetric matrix in $\mathbb{R}^{n\times n}$.

\begin{algorithm}[h]
\caption{Policy Improvement}
\label{alg:policy_improve}
\begin{algorithmic}[1]
\Require $A,\pi,(\eta,x)$.
\Ensure $\pi'$
\State Set $K(i)=\underset{(i,j)\in E}{\arg\max}\;\eta_j$ and $L(i)=\underset{(i,j)\in K(i)}{\arg\max}\;(w(i,j)-\eta_j+x_j)$ for $i=1,\ldots,n$, $I=\{i|\underset{(i,j)\in K(i)}{\max}(w(i,j)-\eta_j+x_j)>x_i\}$, $J=\{i|\underset{(i,j)\in E}{\max}\eta_j > \eta_i\}$.
\If {$I=J=\O$}
\State \Return
\Else
\If {$J\neq \O$}
\State \[\pi'(i)=\left\{\begin{array}
{ll}\text{an arbitrary }e\in K(i) & \text{if }i\in J,\\
\pi(i) &\text{if }i\notin J.
\end{array}\right.\]
\EndIf
\If {$J=\O,I\neq \O$}
\State \[\pi'(i)=\left\{\begin{array}
{ll}\text{an arbitrary }e\in L(i) & \text{if }i\in I,\\
\pi(i) &\text{if }i\notin I.
\end{array}\right.\]
\EndIf
\EndIf
\State \Return
\end{algorithmic}
\end{algorithm}

Last, the Algorithms \ref{alg:value_determin} and \ref{alg:policy_improve} are called to find the eigenmode of $A$. In short, the policy is iteratively improved until situation.
\begin{algorithm}[h]
\caption{Howard's Algorithm: Policy Iteration}
\label{alg:policy_iteration}
\begin{algorithmic}[1]
\Require $A$
\Ensure $\lambda, x, \pi$
\State Initialize policy $\pi$ and compute eigenmode $(\eta,x)$ of $A^{\pi}$ by Algorithm \ref{alg:value_determin}
\While {the itartion is not terminated}
\State Improve $\pi$ by Algorithm \ref{alg:policy_improve}
\State Compute the updated eigenmode $(\eta,x)$ of $A^{\pi}$ by  Algorithm \ref{alg:value_determin}
\EndWhile
\State \Return $\lambda=\eta_1,x,\pi$
\end{algorithmic}
\end{algorithm}
It is proved in \cite{Howard} that the iterations always terminate in finitely many steps, and one iteration requires $O(|E|)$ time. Algorithm \ref{alg:policy_iteration} requires $O(n)$ space.

We use an implementation of Howard's algorithm which is described in the paper  \cite{Howard} as a software library. The source code and more information about this implementation can be found in this site \url{ http://www.cmap.polytechnique.fr/~gaubert/HOWARD2.html}.

\end{document}